\documentclass[12pt,a4paper]{article}
\usepackage{amsmath,euscript}
\usepackage{amssymb}
\usepackage{amsthm}
\usepackage{enumerate}
\usepackage{amsfonts}
\usepackage{comment}
\usepackage{amscd}

\newtheorem{theorem}{Theorem}[section]

\newtheorem{proposition}{Proposition}[section]
\newtheorem{lemma}[theorem]{ Lemma}

\title{Golod-Shafarevich algebras, free subalgebras and Noetherian images}

\author{Agata Smoktunowicz}

\date{ }

\begin{document}

\maketitle

\begin{abstract}
   It is shown that Golod-Shaferevich algebras of a reduced number of defining relations  contain noncommutative free subalgebras in two generators, and that these algebras can be homomorphically mapped  onto prime, Noetherian algebras with linear growth. It is also shown that Golod-Shafarevich algebras of a reduced number of relations cannot be nil.
       \end{abstract}



\section*{Introduction}

In 1964, Golod and Shaferevich proved the Golod-Shafarevich theorem, and subsequently Golod constructed finitely generated  nil and  not nilpotent algebras \cite{gs, g}.
 Recall that an algebra is nil if every element to some power is zero, and that finitely generated nil algebras have no infinite-dimensional homomorphic images which are Noetherian, nor which satisfy a  polynomial identity \cite{lam}. Therefore, in general we cannot hope that Golod-Shafarevich algebras with an infinite number of defining relations can be mapped onto infinite dimensional Noetherian algebras, nor onto infinite dimensional algebras satisfying a polynomial identity.

   In this paper, we will show that the case where the number of defining relations is finite is different;
  namely, that the following result holds:

 \begin{theorem}
  Let $K$ be an algebraically closed field, and let $A$ be the free noncommutative algebra generated in degree one by elements $x, y$. Let $\xi $ be a natural number. Let
   $I$ denote the ideal generated in $A$ by homogeneous
    elements $f_{1}, f_{2}, \ldots , f_{\xi }\in A$. Suppose that there are exactly $r_{i}$
    elements among $f_{1}, f_{2}, \ldots , f_{\xi}$ with degrees larger than
 $2^{i}$ and not exceeding $2^{i+1}$. Assume that there are no elements among $f_{1}, f_{2}, \ldots , f_{\xi }$ with degree $k$ if
    $2^{n}+2^{n-1}+2^{n-2}<k<2^{n+1}+2^{n}$ for some $n$. Denote $Y=\{n:r_{n}\neq 0\}$. Suppose that for all $n\in Y$, $m\in
    \{0\}\cup Y$ with $m<n$ we have
 $$2^{3n+4}\prod_{i<n, i\in Y}r_{i}^{32}< r_{n}< 2^{2^{n-m-3}}.$$
 Then $A/I$ contains a free noncommutative graded subalgebra
    in two generators, and these generators are monomials of the same degree. In particular, $A/I$ is not Jacobson
    radical. Moreover, $A/I$ can be homomorphically mapped onto a
    graded, prime, Noetherian algebra with linear growth which satisfies a polynomial identity.
 \end{theorem}
  For a more general result see Theorem $7.1$.
  The following related  question remains open: {\em Is there a finitely presented infinite dimensional nil algebra?} Note that, under the conditions of Theorem $0.1$, the answer to this question is negative. It was shown by Zelmanov that this question is strongly related to the Burnside problem for finitely presented groups, namely: {\em Are all finitely presented torsion groups finite?}
    It is also not known if the Jacobson radical of a finitely presented algebra is nil. Zelmanov also asked
    whether an algebra in d generators subject to $d^{2}\over 4$ relations can be mapped onto infinite dimensional polynomial identity algebras (with an affirmative answer possibly having applications within group theory \cite{zel}).
 This question is related to Theorem $0.1$ since, by the Small-Warfield result \cite{sw}, affine algebras with linear growth are finite dimensional modules over their centers, with these centers being Noetherian and containing
    subalgebras isomorphic to the polynomial ring $K[x]$.

 We recall the definition of the aforementioned Golod-Shafarevich theorem,
 see~\cite{gs, zel, sb}.
Golod-Shafarevich proved that if the series $(1-dt+\sum_{i=2}^{\infty }r_{i}t^{i})^{-1}$ has all coefficients nonnegative then all free algebras in d generators subject to arbitrary homogeneous relations $f_{1}, f_{2}, \ldots $ with $r_{i}$ relations of degree $i$, are infinite dimensional. In \cite{an2} Anick asked if the converse of the Golod-Shafarevich theorem is true i.e., if there is a finitely generated algebra in d generators subject to $r_{i}$ homogeneous  relations of degree $i$ for $i=2,3,\ldots $, provided that  the series $(1-dt+\sum_{i=2}^{\infty }r_{i}t^{i})^{-1}$
has a negative coefficient. This question is still generally open, however the case of relations of degree two is well understood \cite{wis}, and
 the complete solution in the case of quadratic semigroup relations was found in \cite{is}. Another open question, again related to Theorem $0.1$, is whether finitely presented algebras with exponential growth always contain free noncommutative subalgebras.
 Theorem $0.1$ shows that, under its assumptions, the answer is in the affirmative.
 Anick proved that finitely presented monomial algebras with exponential growth always contain
  free noncommutative subalgebras, and recently  Bell and Rogalski proved that quotients of affine domains
   with Gelfand-Kirillov dimension two over uncountable, algebraically closed fields contain free noncommutative subalgebras in two generators \cite{br}. An open question by Anick asks whether all division algebras of exponential growth contain free noncommutative subalgebras in two generators \cite{an1}. Related questions concerning Golod-Shafarevich groups have also been studied \cite{e, zel}.
 In particular, Zelmanov  proved that a pro-p group satisfying the Golod-Shafarevich condition contains a free non abelian pro p-group \cite{ze}.

\subsection{Notation}
In what follows, $K$ is a countable, algebraically closed field and
$A$ is the free $K$-algebra in two non-commuting indeterminates $x$
and $y$. By a graded algebra we mean an algebra graded by the additive semigroup of natural numbers.
The set of monomials in $\{x,y\}$ is denoted by $M$ and, for
each $k\geq 0$, its subset of monomials of degree $k$ is denoted by
$M(k)$.  Thus, $M(0)=\{1\}$ and for $k\geq1$ the elements in $M(k)$
are of the form $x_1\cdots x_k$ with all $x_i\in \{x, y\}$. The span
of $M(k)$ in $A$ is denoted by $A(k)$; its elements are called
\emph{homogenous polynomials of degree $k$}. More generally, for any
subset $X$ of $A$, we denote by $X(k)$ its subset of homogeneous
elements of degree $k$. The \emph{degree} $\deg f$ of an element $f \in A$ is the least
$k\ge0$ such that $f \in A(0) + \cdots + A(k)$. Any $f\in A$ can be
uniquely written in the form $f=f_0+f_1+\cdots+f_k$ with each $f_i\in
A(i)$. The elements $f_i$ are the \emph{homogeneous components} of
$f$.  A (right, left, two-sided) ideal of $A$ is \emph{homogeneous} if
it is spanned by its elements' homogeneous components.  If $V$ is a
linear space over $K$, we denote by $\dim V$ the dimension of $V$
over $K$.
 A graded $K$-algebra $R$ has a linear growth if there is a number $c$ such that $\dim R(n)\leq c$
 for all $n$. We say that a graded infinite-dimensional algebra $R$ has quadratic growth if there is a number $c$ such that
   $\dim R(n)\leq cn$
 for all $n$, and $R$ does not have linear growth. For more information about the growth of algebras, see \cite{kl}.
\section{General construction}
Let $K$ be a field and $A$ be a free $K$-algebra generated in degree one by two elements $x$, $y$.
 Suppose that subspaces $U(2^m),V(2^m)$ of $A(2^m)$ satisfy,
  for every $m\geq 1$, the following properties:
  \begin{itemize}
  \item[1.] $V(2^m)$ is spanned by monomials;\label{prop:3}
  \item[2.] $V(2^m)+U(2^m)=A(2^m)$ and $V(2^m)\cap U(2^m)=0$;
  \item[3.] $A(2^{m-1})U(2^{m-1})+U(2^{m-1})A(2^{m-1})\subseteq U(2^m)$;
  \item[4.] $V(2^m)\subseteq V(2^{m-1})V(2^{m-1})$, where for $m=0$ we set $V(2^{0})=Kx+Ky$, $U(2^{0})=0$.
  \end{itemize}
We define a graded subspace $ E$ of $A$
by constructing its homogeneous components $ E(k)$ as
follows. Given $k\in N$, let $n\in N$ be such that $2^{n-1}\leq
k<2^n$. Then $r\in E(k)$ precisely if, for all
$j\in\{0,\dots,2^{n+1}-k\}$, we have $A(j)rA(2^{n+1}-j-k)\subseteq
U(2^n)A(2^n)+A(2^n)U(2^n)$. More compactly,
\begin{equation}
   E(k)=\{r\in A(k)\mid ArA\cap A(2^{n+1})\subseteq U(2^n)A(2^n)+A(2^n)U(2^n)\}.
\end{equation}
Set then $ E=\bigoplus_{k\in N} E(k)$.

\begin{lemma}
  The set $ E$ is an ideal in $A$. Moreover, if all sets $V(2^{n})$ are nonzero, then
 algebra $A/ E$ is infinite dimensional over $K$.
\end{lemma}
\begin{proof} The proof of the first claim is the same as in Theorem $5$ in \cite{ls} ($A(n)$ is denoted as $H(n)$ in \cite{ls}). Notice that we only need property 3 to prove that $E$ is an ideal.

  Regarding the second claim, suppose on the contrary that $A/E$ is a finite-dimensional algebra.
 By the definition of $E$, we see that $A/E$ is a graded algebra, hence $V(2^n)\subseteq  E$ for some $n$.
 Let $r\in V(2^{n})\subseteq E$, then by the definition of $E$ we get that
  $rA(2^{n+2}-2^{n})\subseteq U(2^{n+2})$. Since $V(2^{n})^{3}\subseteq A(2^{n+2}-2^{n})$, it follows that $V(2^{n})^{4}\subseteq U(2^{n+2})$.
   By property $4$, $V(2^{n+2})\subseteq V(2^{n})^{4}$, hence
 $V(2^{n+2})\subseteq U(2^{n+2})$. By property $2$, $U(2^{n+2})\cap V(2^{n+1})=0$, a contradiction,
 hence $A/E$ is infinite dimensional over $K$.
\end{proof}
 We now prove the following theorem.
\begin{theorem} Let $K$ be a field and $A$ be a free $K$-algebra generated in degree one by two elements $x, y$.
  Suppose that subspaces $U(2^m),V(2^m)$ of $A(2^m)$ satisfy properties 1-4 above,
  and moreover that  there is $n$ such that $\dim V(2^{n})=2$ and $V(2^{m+1})=V(2^{m})V(2^{m})$ for all $m\geq n$.
  Then, the algebra $A/E$ contains a free noncommutative algebra in $2$ generators, and these generators are monomials of the same degree.
\end{theorem}
\begin{proof} Let $V(2^{n})=Km_{1}+Km_{2}$ for some monomials $m_{1}, m_{2}\in A(2^{n})$.
 We will show that images of $m_{1}$ and $m_{2}$ generate a free noncommutative subalgebra in $A/E$.
 Recall that $E$ is a graded ideal; therefore, it is sufficient to show that if $f(X,Y)\in K[X,Y]$ is a homogeneous polynomial in two noncommuting variables $X, Y$,
 then $f(m_{1}, m_{2})\notin E$. Notice that $f(m_{1}, m_{2})\in A(t)$ for some $t$ divisible by $2^{n}$. Let $m$ be such that  $2^{m}\leq t< 2^{m+1}$ and let $j={{2^{m+2}-t}\over 2^{n}}$.
  Observe that $f(m_{1}, m_{2})\in V(2^{n})^{t\over 2^{n}}$, since $m_{1}, m_{2}\subseteq V(2^{n})$, and so $f(m_{1}, m_{2}){m_{1}}^{j}\in V(2^{n})^{2^{m+2-n}}$. Observe that $V(2^{n})^{2^{m+2-n}}=V(2^{m+2})$, since by assumption $V(2^{m+1})=V(2^{m})V(2^{m})$ for all $m\geq n$. By property $3$, we get  $U(2^{m+2})\cap V(2^{m+2})=0$. Therefore, $f(m_{1}, m_{2}){m_{1}}^{j}\notin U(2^{m+2})$, and so $f(m_{1}, m_{2})\notin E$, as required, by the definition of $E$. This completes the proof.
\end{proof}
\begin{theorem}
Let $K$ be a field and $A$ be a free $K$-algebra generated in degree one by two elements $x, y$.
  Suppose that subspaces $U(2^m),V(2^m)$ of $A(2^m)$ satisfy properties 1-4 above and moreover that
  there is $\alpha $ such that $\dim V(2^{m})=1$ for all $m\geq \alpha $. Then $V(2^{m+1})=V(2^{m})V(2^{m})$
  and $U(2^{m+1})=U(2^{m})A(2^{m})+A(2^{m})U(2^{m})$ for all $m\geq n$.
\end{theorem}
\begin{proof}
Observe that by property $4$, $V(2^{m+1})\subseteq V(2^{m})V(2^{m})$, and by assumption $\dim V(2^{m+1})=\dim V(2^{m})V(2^{m})$ for all $m\geq \alpha $. Therefore, $V(2^{m+1})= V(2^{m})V(2^{m})$.

Fix $m\geq \alpha $. We will now prove that $U(2^{m+1})= T$ where $T=U(2^{m})A(2^{m})+A(2^{m})U(2^{m})$. Notice that $T\subseteq U(2^{m+1})$ by property 3. On the other hand, $V(2^{m+1})=V(2^{m})V(2^{m})$ and property $2$ imply that $T+V(2^{m+1})=A(2^{m+1})$.
 Therefore, $\dim T=\dim A(2^{m+1})-\dim V(2^{m+1})$. By property $2$, $\dim U(2^{m+1})=\dim A(2^{m+1})-\dim V(2^{m+1})$, so $\dim T=\dim U(2^{m+1})$. Since $T\subseteq U(2^{m+1})$, it follows that $U(2^{m+1})=T$.
 \end{proof}
We will now review some concepts introduced in \cite{ls}.
  We will adhere to the notation used in \cite{sb}.

We extend the definition of $U(2^n),V(2^n)$ to dimensions that are
not powers of $2$. In ~\cite{ls}, Section $4$ the sets~(\ref{def:U<}--\ref{def:V>}) are named
respectively $S, W, R$ and $Q$.

Let $k\in N$ be given. We write it as a sum of increasing powers of $2$,
namely $k=\sum_{i=1}^t 2^{p_i}$ with $0\leq p_1 < p_2 < \ldots <
p_t$. Set then
\begin{align}
  U^<(k) &= \sum_{i=0}^t A(2^{p_1}+\cdots+2^{p_{i-1}})U(2^{p_i})A(2^{p_{i+1}}+\cdots+2^{p_t}),\label{def:U<}\\
  V^<(k) &= V(2^{p_1})\cdots V(2^{p_t}),\\
  U^>(k) &= \sum_{i=0}^t A(2^{p_t}+\cdots+2^{p_{i+1}})U(2^{p_i})A(2^{p_{i-1}}+\cdots+2^{p_1}),\\
  V^>(k) &= V(2^{p_t})\cdots V(2^{p_1}).\label{def:V>}
\end{align}
\begin{lemma}[\cite{ls}, pp. 993--994]
  For all $k\in N$ we have $A(k)=U^<(k)\oplus V^<(k)=U^>(k)\oplus
  V^>(k)$.

  For all $k,\ell\in N$ we have $A(k)U^<(\ell)\subseteq U^<(k+\ell)$ and $U^>(k)A(\ell)\subseteq U^>(k+\ell)$.
\end{lemma}
\begin{proposition}[Theorem 11,\cite{ls}]
  For every $k\in N$ we have
  \[\dim A(k)/ E(k)\leq \sum_{j=0}^k\dim V^<(k-j)\dim V^>(j),\]
 where we set $\dim V^{>}(0)=\dim V^{<}(0)=1$.
\end{proposition}
\begin{proof} The proof is the same as the proof of Theorem $11$ in \cite{ls}, or the proof of Theorem $5.2$ in \cite{lsy}.
\end{proof}
\begin{lemma}[Lemma 3.2, \cite{lsy}]
 For any $m\geq n$ and any $0\leq  k < 2^{m-n}$,
$$H(k2^{n})U(2^{n})H((2^{m-n} - k - 1)2^{n})\subseteq U(2^{m}).$$
\end{lemma}
\section{Algebras satisfying polynomial identities}
 Throughout this section, $K$ is a  field and $A$ is a free $K$-algebra generated in degree $1$ by two elements $x$, $y$. Also in this section we assume that subspaces $U(2^m),V(2^m)$ of $A(2^m)$ satisfy properties 1-4 from the beginning of Section $1$, and that
  there is $\alpha $ such that $\dim V(2^{m})=1$ for all $m\geq \alpha $.
\begin{lemma}
 There is a natural number $c$, such that for every $n$ the dimension of the space $T_{n}=\sum_{i=0}^{n}V^<(i)V^>(n-i)$ is less than $c$.
\end{lemma}
\begin{proof} By assumption, there is a monomial $v$ of degree $2^{\alpha }$ such that $V(2^{\alpha })=Kv$.
 Observe that by assumption $V(2^{\alpha +i})=Kv^{2^i}$. By the definition, the spaces $V^<(i)$ and $V^>(i)$ are contained in the appropriate products of spaces $V(2^{k})$.
 It follows that the space $T_{n}$ has a basis consisting of elements of the form
  $qv^{i}r$ where $q$ and $r$ are monomials of degrees not exceeding $2^{\alpha }$, and $i$ is such that the total degree of $qv^{i}r$ is $n$.
   It follows that $\dim T(n)\leq 2^{2^{2\alpha +2}}$, as required.
\end{proof}
We now recall the definition of the Capelli Polynomial (see pp. 8, \cite{br}; pp. 141, \cite{df}).
 $$d_{n}(x_{1}, x_{2}, \ldots , x_{n}; y_{1}, y_{2}, \ldots , y_{n})=\sum_{\sigma \in S\{1,2, \ldots ,n\}}(-1)^{\sigma }x_{\sigma (1)}y_{1}x_{\sigma (2)}y_{2}\ldots x_{\sigma (n)}y_{n}$$
  where $S\{1, 2, \ldots ,n\}$ is the set of all permutations of the set $\{1, 2, \ldots , n\}$.
\begin{lemma} Let $c$ be as in Lemma $2.1$. Let $n$ be a natural number and $m_{1}, \ldots , m_{c+1}\in M(n)$, and  $r_{1}, r_{2}, \ldots , r_{c+1}\in M$ be such that
 $\deg (r_{i})+n$ is divisible by $2^{\alpha }$ for all $i\leq c+1$. Then
 $$d_{c+1}(m_{1}, m_{2}, \ldots , m_{c+1}; r_{1}, r_{2}, \ldots , r_{c+1}) \in E.$$
\end{lemma}
\begin{proof} Denote $P=d_{c+1}(m_{1}, m_{2}, \ldots , m_{c+1}; r_{1}, r_{2}, \ldots , r_{c+1})$.
Observe that $P\in A(\gamma)$ where $\gamma=2^{\alpha }q$ for some $q$.
Let $s$ be such that $2^{s}\leq \gamma <2^{s+1}$, so clearly $\alpha <s+1$. We will show that, for any $t>0$ and any $0<i\leq 2^{\alpha }$, we have
 $A(2^{\alpha }t-i)PA\cap A(2^{s+2})\subseteq U(2^{s+1})A(2^{s+1})+A(2^{s+1})U(2^{s+1})$.
  Then, because $i, t$ were arbitrary, and by the definition of $E$, we would get $P\in E$.\\
 Fix $0<i \leq 2^{\alpha }$.  We will show now that $A(2^{\alpha }t-i)PA\cap A(2^{s+2})\subseteq U(2^{s+1})A(2^{s+1})+A(2^{s+1})U(2^{s+1})$.
 By Lemma $1.3$, $U^{<}(i)A(n-i)+A(i)U^{>}(n-i)+V^{<}(i)V^{>}(n-i)=A(n)$.
  By Lemma $2.1$, there are $\beta _{i}\in K$ such that $\sum_{j=1}^{c+1}\beta_{j}m_{j}\in
 U^{<}(i)A(n-i)+A(i)U^{>}(n-i)$. We can assume that $m_{c+1}=\sum_{i\leq c}m_{i}\gamma _{i}+d$ for some $\gamma _{i}\in K$
 and some $d\in U^{<}(i)A(n-i)+A(i)U^>(n-i)$. After substituting the expression for $m_{c+1}$ into the expression for $P$, we get $$P\in \sum_{k=0,1,2, \ldots }A(2^{\alpha }k)(U^{<}(i)A(n-i)+A(i)U^{>}(n-i))A(2^{\alpha }(q-k)-n).$$

   Note that, since $\deg P$ is divisible by $2^{\alpha }$, we get $A(2^{\alpha }t-i)PA\cap A(2^{s+2})\subseteq A(2^{\alpha }t-i)PA(i)A$.
   Therefore, $A(2^{\alpha }t-i)PA\cap A(2^{s+2})\subseteq \sum_{k=0,1,2, \ldots }A(2^{\alpha }(k+t-1))A(2^{\alpha }-i)(U^{<}(i)A(n-i)+A(i)U^{>}(n-i))A(2^{\alpha }(q-k)-n)A(i)A.$

Observe now that the following holds:
\begin{itemize}
\item[a.] By Lemma $1.3$, we get $A(2^{\alpha }-i)U^{<}(i)\subseteq U(2^{\alpha })$. Recall that $U(2^{\alpha })=U^{>}(2^{\alpha })$. Hence, $A(2^{\alpha }(k+t-1))A(2^{\alpha }-i)(U^{<}(i)A(n-i))A(2^{\alpha }(q-k)-n)A(i)A \subseteq A(2^{\alpha }(k+t-1))U^{>}(2^{\alpha })A$.
\item[b.] By Lemma $1.3$,  we get $U^{>}(n-i)A(2^{\alpha}(q-k)-(n-i))\subseteq U^{>}(2^{\alpha}(q-k))$. Therefore, $A(2^{\alpha }(k+t-1))A(2^{\alpha }-i)(A(i)U^{>}(n-i))A(2^{\alpha }(q-k)-n)A(i)A\subseteq A(2^{\alpha }(k+t))U^{>}(2^{\alpha}(q-k))A.$
\end{itemize}

 Using a. and b. we get
$$A(2^{\alpha }t-i)PA\cap A(2^{s+2})\subseteq \sum_{k,j=0,1,2, \ldots }A(2^{\alpha }k)U^{>}(2^{\alpha}(j+1))A.$$

 By assumptions from the beginning of this section, we get $\dim V(2^{\alpha +i})=1$ if $i\geq 0$. By Theorem $1.3$, $U(2^{\alpha +i+1})=U(2^{\alpha +i})A(2^{\alpha +i})+A(2^{\alpha +i})U(2^{\alpha +i})$
   for all $i\geq 0$. Applying this property several times, we get that for all natural $j>0$ we have
   $U^{>}(2^{\alpha }j)=\sum_{i=0,1, \ldots , j-1}A(2^{\alpha }i)U(2^{\alpha })A(2^{\alpha }(j-i-1))$. It follows that
   $A(2^{\alpha }t-i)PA\cap 2^{s+2}\subseteq
  \sum_{i,j =0,1,\ldots }A(2^{\alpha }i)U(2^{\alpha })A(2^{\alpha }j)$. Therefore, by Lemma $1.4$ and because $s+1> \alpha $,
    we get that  $A(2^{\alpha }t-i)PA\cap A(2^{s+2})\subseteq A(2^{s+1})U(2^{s+1})+A(2^{s+1})U(2^{s+1})$.
\end{proof}

\begin{lemma} Let $c$ be as in Lemma $2.1$. Let $\beta > 2^{\alpha }(c+1)$.
 Let $n$ be a natural number, $m_{1}, \ldots , m_{\beta}\in M(n)$ and let $r_{1}, r_{2}, \ldots , r_{\beta}\in M$.
  Then $$d_{\beta }(m_{1}, m_{2}, \ldots ,m_{\beta }; r_{1}, r_{2}, \ldots , r_{\beta})\in E.$$
\end{lemma}
\begin{proof}  Denote $Z=d_{\beta }(m_{1}, m_{2}, \ldots ,m_{\beta }; r_{1}, r_{2}, \ldots , r_{\beta})$. We will show that $Z$ is in the ideal generated by elements of the same form as $P$, in Lemma $2.2$.
 Consider elements $e(1)=\deg r_{1}+n$, $e(2)=\deg r_{1}+\deg r_{2}+2n,$ \ldots , $e(i)=\sum_{k=1}^{i}\deg r_{k}+ni$.
  There are $c+1$ elements $i_{1}, i_{2}, \ldots ,i_{c+1}$ such that $e(i_{1}), e(i_{2}), \ldots , e(i_{c+1})$ give the same remainder modulo $2^{\alpha }$.
   Denote $Z(t, i)=d_{\beta -1}(m_{1}, \ldots ,m_{i-1},$ $ m_{i+1},$ $ \ldots ,m_{\beta };
   r_{1}, \ldots ,r_{i-2}, r_{t-1}m_{i}r_{t}, r_{i+1}, \ldots , r_{\beta}).$
  Observe that, for every $t\leq n$,  $Z$ is a linear combination of such elements, namely $Z\in \sum_{i=1, \ldots ,\beta }KZ_{i,t}$. We will call this construction specializing at place $t$.
 We can repeat this construction to expressions $Z_{i,t}$ instead of $Z$. After repeating this construction several times and specializing at suitable places, we get that $Z$ is a linear combination of elements of the form
 $q(d_{c+1}(M_{1}, \ldots ,M_{c+1}; q_{1}, \ldots ,q_{c+1}))$ where $\{M_{1}, \ldots , M_{c+1}\}\subseteq \{m_{1}, \ldots , m_{\beta }\}$, and $q\in A$, $q_{1}, \ldots , q_{c+1}\in M$ are such that $n+q_{i}$ is divisible by $2^{\alpha }$ for each $i\leq c+1$, and $q\in A$. By Lemma $2.2$, all such elements are in $ E$, and hence $Z\in E$.
\end{proof}
\begin{lemma} Let $K$ be a field, and
let $R$ be an infinite-dimensional, graded, finitely generated $K$-algebra.
Then $R$ can be homomorphically mapped  onto a prime, infinite-dimensional,  graded algebra.
 Moreover, if $R$ has quadratic growth and satisfies a polynomial identity, then $R$ can be homomorphically mapped onto a prime, graded algebra with  linear growth.
\end{lemma}
\begin{proof} We first construct a prime homomorphic image of $R$. Let $B(R)$ be the prime radical of $R$, then $R/B(R)$ is semiprime. Observe that $B(R)$ is homogeneous, since $R$ is graded. Therefore, $R/B(R)$ is graded. Consequently, as a graded ring,  $R/B(R)$ is either  infinite dimensional or nilpotent. It cannot be nilpotent, because it is semiprime. Hence $P=R/B(R)$ is infinite-dimensional, graded and semiprime. Note that since the prime radical of $P$ is zero, the intersection of all prime ideals in $P$ (which equals the prime radical) is zero, hence there is a prime ideal $Q$ in $P$ which is not equal to $P$. Note that the largest homogeneous subset, call it $M$, contained in $Q$ is also a prime ideal in $P$. Now, $P/M$ is prime and non-zero and graded, hence it is infinite dimensional, as required.

 Suppose now that $R$ satisfies a polynomial identity. We will now show that $P$ can be homomorphically mapped onto a prime, graded algebra with linear growth. Since P is prime
 and satisfies a polynomial identity, by Rowen's theorem \cite{row} it has a non-zero central element $z=z_1+...+z_m$, where
 $z_i$ has degree $i$ and $z_m$ is nonzero and $m\geq 1$.
  Notice that $z_m$ is central, since if r is a homogeneous element of degree $d$
 then $[r,z] = [r,z_m]+ e$ where $e$ consists of lower degree terms.
 Observe that  $z_m$ is regular as P is prime.

 Let $S=P/z_mP$.  Then $\dim S(n)=\dim P(n)-\dim P(n-m)$, since $z_m$ is
 regular, homogeneous and of degree  $m$.
 Since $P$ is graded with quadratic growth we have some $C$ such that $\dim P(n)<Cn$ for all $n$.
  Thus we have
$ \sum_{i=0}^{n}\dim S(i)<\sum_{i=0}^{n}\dim P(n)-\dim P(n-m) <\sum_{i=0}^{m}\dim P(n-i)\leq  Cmn$.
By \cite{sw}, and since $m$ and $C$ are constant, we see that $S$ has at most linear growth.

 If $P$ has linear growth then the proof is finished, as $P$ has linear growth and is a homomorphic image of $R$.

 If $P$ has greater than linear growth then $S$ is infinite dimensional, and so by the first
 part of our theorem $S$ has prime infinite dimensional image, with linear growth (which is also a homomorphic image of $R$).
\end{proof}
\begin{theorem}
 The algebra $A/E$ can be homomorphically mapped onto a graded, Noetherian, prime algebra with linear growth.
\end{theorem}
\begin{proof} By Lemma $2.4$, $A/E$ has a prime, graded, infinite dimensional image which is graded by natural numbers; call it $P$.
We will now show that $P$ satisfies a polynomial identity.
 The extended centroid of a prime ring is a field (see page 70, line 16 \cite{bfm}). Let $C$ be the extended centroid of $P$, then the central closure
 $CP$ of $P$  has at most linear growth (as an algebra over the field $C$), as there is less than $\beta n$ elements of degree not exceeding $n$ linearly independent over $C$, by Lemma $2.3$ and Theorem $2.3.7$ \cite{bfm}. By the Small-Stafford-Warfield theorem, every algebra with linear growth satisfies a polynomial identity, therefore $CP$ satisfies a polynomial identity \cite{ssw}.
  It is known that the extended closure $CP$ of a prime ring $P$ is prime (see \cite{bcm}, pp. 238), and hence by \cite{sw} the algebra  $CP$ is finite-dimensional over its center. By Lemma $1.21$ in \cite{br}, we see that $CP$ satisfies a Capelli identity, therefore $P$ is a polynomial identity algebra.

 Let $c$ be as in Lemma $2.1$. By Proposition $1.1$, we have $\dim A(k)/ E(k)\leq \sum_{j=0}^k\dim V^<(k-j)\dim V^>(j)<(k+1)c,$ and so $A/E$ has at most quadratic growth.
Notice that $P$ has growth smaller than $A/E$. By Bergman's Gap theorem, $P$ has  either linear or quadratic growth. If the latter holds then we are done. Suppose that $A/E$
 has quadratic growth; then by Lemma $2.4$, $P$ has a homomorphic image with linear growth. By the Small-Warfield theorem \cite{sw}, prime, finitely generated algebras with linear growth are Noetherian; this completes the proof.

\end{proof}

 \section{Constructing algebras satisfying given relations}

  Here we give the criteria for when an element $f\in A$ is in the ideal $E$.
 \begin{theorem} Let $n$ be a natural number. Suppose that subspaces $U(2^m),V(2^m)$ of $A(2^m)$ satisfying properties 1-4 (from Section $1$) were constructed for all $n$.
 Let $f\in A(k)$, where $2^{n}\leq k< 2^{n+1}$. Suppose that
   $$ AfA\cap A(2^{n+1})\subseteq A(2^n)U(2^n)+U(2^n)A(2^n).$$
  Suppose, moreover, that for all $i,j\ge0$ with $i+j=k-2^n$ we have $f\in
    A(i)U(2^n)A(j)+U^<(i)A(k-i)+A(k-j)U^>(j)$, with the sets
    $U^<(i),U^>(i)$ defined as in Section $1$.  Then $f\in  E$.
\end{theorem}
\begin{proof} To show that $f\in E$, it suffices to prove that
  $A(i)fA(j)\subset T:=A(2^{n+1})U(2^{n+1})+U(2^{n+1})A(2^{n+1})$
  for all $i,j\in N$ with $i+j+k=2^{n+2}$.

  If $i\geq 2^{n+1}$, we get
  $A(i-2^{n+1})fA(j)\subseteq A(2^n)U(2^n)+U(2^n)A(2^n)\subseteq U(2^{n+1})$, so
  $A(i)fA(j)\subseteq T$. Similarly, if $j\geq 2^{n+1}$, we get
  $A(i)fA(j-2^{n+1})\subseteq U(2^{n+1})$, so $A(i)fA(j)\subseteq T$. If
  $i,j\geq 2^n$, then $A(i-2^n)fA(j-2^n)\subseteq A(2^n)U(2^n)+U(2^n)A(2^n)$,
  so $A(i)fA(j)\subseteq T$.

  If $i<2^n$ and $j<2^{n+1}$, then $f\in
  A(2^n-i)U(2^n)A(2^{n+1}-j)+U^<(2^n-i)A(2^{n+1}+2^n-j)+A(2^{n+1}-i)U^>(2^{n+1}-j)$
  by the assumption, so $A(i)fA(j)\subseteq
  A(2^n)U(2^n)A(2^{n+1})+T\subseteq T$, because
  $A(i)U^{<}(2^{n}-i)\subseteq U(2^{n})$ and $U^{>}(2^{n+1}-j)A(j)\subseteq U(2^{n+1})$,
  by Lemma $1.3$. The case $i<2^{n+1},j<2^n$ is
  handled similarly.  We may now conclude that $f=0$ holds in $A/E$.
\end{proof}

\section{Constructing $U(2^n),V(2^n)$ from $F(2^n)$}\label{ss:UV}

 In this section we introduce sets $F(2^{n})$, which will later be used to show that the algebra $A/E$ satisfies given relations.
 Roughly speaking, the relations which we want to hold in $A/E$ will be contained in the sets $F(2^{n})$. For more details about the properties and construction of the sets $F(2^{n})$, see Section $6$.

We begin with a modification of Theorem $3$ from \cite{ls}. Let $r_{n}$ be as in Theorem $0.1$. Let $Y=\{n:r_{n}\neq 0\}$ and a sequence of natural numbers $\{e(n)\} _{n\in Y}$ be given,
\begin{align}
  S&=\bigsqcup_{k\in Y}\{k-e(k)-1,\dots,k-1\}\label{def:S}
\end{align}
and assume that the union defining $S$ is disjoint and $S$ is a subset of natural numbers (we assume that zero is a natural number).
 Let $sup Y$ denote the largest element in $Y$ if $Y$ is finite.
\begin{theorem}\label{8props}

  Let $Y$, $S$ and $\{e(n)\}_{n\in Y}$ be as above. Let an integer $n$ be given.  Suppose
  that, for every $m\leq n$, we are given a subspace $F(2^m)\subseteq
  A(2^m)$ with $\dim F(2^m)\leq (2^{2^{e(m)}})^2-2$ and that, for every
  $m<n$, we are given subspaces $U(2^m),V(2^m)$ of $A(2^m)$ with
  \begin{itemize}
  \item [1.] $\dim V(2^m)=2$ if $m\notin S$ and $m\leq sup Y$ if $Y$ is finite;
  \item [2.]$\dim V(2^{m-e(m)-1+j})=2^{2^j}$ for all $m\in Z$ and all $0\leq
    j\leq e(m)$;\label{prop:2}
  \item [3.] $V(2^m)$ is spanned by monomials;
  \item [4.] $F(2^m)\subseteq U(2^m)$ for every $m\in Y$, and $F(2^{m})=0$
    for every $m\notin Y$;
  \item [5.] $V(2^m)\oplus U(2^m)=A(2^m)$;
  \item [6.] $A(2^{m-1})U(2^{m-1})+U(2^{m-1})A(2^{m-1})\subseteq U(2^m)$;
  \item [7.] $V(2^m)\subseteq V(2^{m-1})V(2^{m-1})$.
\item [8.] Moreover, if $Y$ is finite and $m>sup Y$, then $V(2^{m})=1$.
  \end{itemize}
  Then there exist subspaces $U(2^n),V(2^n)$ of $A(2^n)$ such that the
  extended collection $U(2^m),V(2^m)_{m\leq n}$ still satisfies
  conditions 1-8.
\end{theorem}
\begin{proof}
  For properties $1-7$, the proof is the same as in \cite{ls} when we use $e(n)$ instead of $log(n)$. The detailed proof with the same notation can be found in \cite{sb}. We can use this proof to
 define inductively $V(2^{i})$, $U(2^{i})$ for all $i\leq sup Y$. Denote $t=sup Y$. By definition of $S$ and $Y$, and by property $1$ from Theorem $4.1$, we get that $V(2^{t})=Km_{1}+Km_{2}$ for some $m_{1}$ and $m_{2}$ in $A(2^{t})$. Then define $V(2^{t+1})=Km_{1}m_{1}$, $U(2^{t+1})=(U(2^{t})+Km_{2})A(2^{t})+
 A(2^{t})(U(2^{t})+Km_{2})$.
  Now define inductively for all $i>0$, $V(2^{t+i+1})=V(2^{t+i})V(2^{t+i})=Km_{1}^{2^{i+1}}$ and
 $U(2^{t+i+1})=U(2^{t+i})A(2^{t+i})+A(2^{t+i})U(2^{t+i})$.
  In this way we constructed sets $U(2^{n}), V(2^{n})$ for all $n>t=sup Y$, satisfying property 8, and properties 3, 5, 6, 7. We set $F(2^{m})=0$ for all $m>t$, so property 4 holds. The properties 1, 2 don't apply for $m>t$. Recall that we already constructed sets $V(2^{i})$, $U(2^{i})$ for all $i\leq sup Y$, using the same proof as in  \cite{ls} or \cite{sb}. The proof is finished.
 \end{proof}
\begin{theorem} The above theorem is also true  when, instead of property $8$, we put the following property:
\begin{itemize}
  \item [8'.] If $Y$ is finite and $k\geq sup Y$, then $V(2^{k+i+1})=V(2^{k+i})V(2^{k+i})$ for all $i\geq 0$.
\end{itemize}
\end{theorem}
\begin{proof} Define inductively $V(2^{i})$, $U(2^{i})$ for all $i\leq sup Y$ as in Theorem $4.1$. Denote $t=sup Y$. By property $1$ from Theorem $4.1$, $V(2^{t})=Km_{1}+Km_{2}$ for some $m_{1}$ and $m_{2}$
 in $A(2^{t})$.  Now define inductively, for all $i\geq 0$,  $V(2^{t+i+1})=V(2^{t+i})V(2^{t+i})$ and
 $U(2^{t+1+1})=U(2^{t+i})A(2^{t+i})+A(2^{t+i})U(2^{t+i})$. In this way we constructed sets $U(2^{n}), V(2^{n})$ for all $n>t=sup Y$, satisfying property 8', and properties 3, 5, 6, 7. We set $F(2^{m})=0$ for all $m>t$, so property 4 holds. The properties 1, 2 don't apply for $m>t$. Recall that we already constructed sets $V(2^{i})$, $U(2^{i})$ for all $i\leq sup Y$, using the same proof as in  \cite{ls} or \cite{sb}. The proof is finished.
 \end{proof}
\section{Growth of subspaces }

In this section we generalize results from Section $2$ in \cite{sb}.
  To lighten
notation, we write $[X]=\dim X$ for the dimension of a
subspace $X\subseteq A$.
Suppose that sets $V(2^{n})$, $U(2^{n})$, $F(2^{n})$ satisfy properties $1-8$ of Theorem $4.1$, with $\{e(i)\}_{i\in Y}$, $S$, $Y$ defined as in Section $4$.
 The results from this section will be mainly used in Section $6$. We begin with a lemma about the dimensions $V^>(k)$ and $V^<(k)$,
continuing with the notation from \cite{sb}.
\begin{lemma}
  Let $\alpha $ be a natural number with binary decomposition $\alpha
  = 2^{p_1} + \cdots +2^{p_{t}}$. Suppose $p_i\notin S$, for all
  $i=1,\ldots, t$. Then $[V^>(\alpha)]\leq 2\alpha$.
 \end{lemma}
\begin{proof} The same as the proof of Lemma $2.1$ in \cite{sb}, but we repeat it for the convenience of the reader.
  If $p_i\notin S$, then $[V(2^{p_i})]\leq 2$ by assumption, so
  \[[V^>(\alpha)]=\prod_{i=1}^t[V(2^{p_i})]=2^t\leq 2^{\log(\alpha)+1}\leq 2\alpha.\qedhere\]
\end{proof}

\begin{lemma}\label{lem:3.2}
  Let $\alpha$ be a natural number with binary decomposition $\alpha
  = 2^{p_1} + \cdots +2^{p_{t}}$. Suppose that there is $k\in Y$ such
  that $p_i\in\{k-e(k)-1,\dots, k-1\}$ for all $i=1,\dots,t$. Then
  $[V^>(\alpha)]\leq 2^{2^{e(k)+1}}$. More precisely, $[V^>(\alpha)]=2^{\alpha/2^{k-e(k)-1}}$.
\end{lemma}
\begin{proof} The same as the proof of Lemma $2.2$ in \cite{sb}. Recall that, by Theorem~\ref{8props}(2), we have
  $[V(2^i)]=2^{2^{i-(k-e(k)-1)}}$ for all
  $i\in\{k-e(k)-1,\dots,k-1\}$. Then
  \begin{align*}
    \log[V^>(\alpha)]&=\log\prod_{i=1}^t[V(2^{p_i})]=\log\prod_{i=1}^t2^{2^{p_i-(k-e(k)-1)}}\\
    &= \sum_{i=1}^t2^{p_i-(k-e(k)-1)}=\frac{\alpha}{2^{k-e(k)-1}}\\
    &\leq 2^{e(k)+1}.\qedhere
  \end{align*}
\end{proof}

\begin{proposition}\label{prop:bdV>}
  Let $\alpha = 2^{p_1} + \cdots + 2^{p_{t}}$ be a natural number in the binary form.
   Then $[V^>(\alpha)]<2\alpha \prod_{i\leq m, i\in Y }2^{2^{e(i)+1}}$, where $m$ is maximal
  such that $\sum_{p_i\in \{m-e(m)-1,\dots, m-1\}}2^{p_{i}}$ is nonzero.
  \end{proposition}
\begin{proof}
  Write $\alpha = 2^{p_1} + \cdots + 2^{p_{t}}$ in binary. Write again
  $S_k=\{k-e(k)-1,\dots,k-1\}$. For all $k\in Y$, set
  $\alpha_k=\sum_{p_i\in S_k}2^{p_i}$. Let $m$ be maximal such that
  $\alpha _{m}\neq 0$.
   Set $\gamma=\sum_{k\leq m}\alpha_k$ and
  $\delta=\sum_{p_i\notin S}2^{p_i}$ so that
  $\alpha=\gamma+\delta $. By definition of the sets $V^>(i)$, we have
  $[V^>(\alpha)]=[V^>(\gamma)][V^>(\delta)]$.
    By Lemma $5.2$,
  $$
  [V^>(\gamma)]= \prod_{k\leq m, k\in Y}[V^>(\alpha_k)]<
  \prod_{k\leq  m, k\in Y}2^{2^{e(k)+1}}.
  $$
  Finally, by Lemma $5.1$, we have $[V^>(\delta)]\leq
  2\alpha$. Putting everything together, we get
  $[V^>(\alpha)]<2\alpha \prod_{i\leq m, i\in Y}2^{2^{e(i)+1}}.$
\end{proof}
\begin{lemma}\label{lem:3.4}
  Let $\alpha, \beta $ be natural numbers such that $\alpha+\beta \leq
  2^{n-1}+2^{n-2}$, for some $n\in Y$. Then
 $$[V^<(\alpha)][V^>(\beta)]\leq 2^{2n}(\prod_{k<n, k\in Y}2^{2^{e(k)+1}})^{2}[V(2^{n-1})]^{2}/2^{2^{e(n)-1}}.$$
\end{lemma}
\begin{proof}
  Write $\alpha = 2^{p_1} + \cdots +2^{p_t}$ in binary. Write again
  $S_k=\{k-e(k)-1,\dots,k-1\}$ and $\alpha_k=\sum_{p_i\in
    S_k}2^{p_i}$. Set now $\gamma=\sum_{k<n}\alpha_k$ and
  $\delta=\sum_{p_i\notin S}2^{p_i}$; we get
  $\alpha=\gamma+\delta+\alpha_n$, and by definition of the sets
  $V^>(n)$, we get $[V^>(\alpha)]=[V^>(\gamma)][V^>(\delta)][V^>(\alpha_n)]$. By previous Lemmas,
  \[[V^>(\gamma)]= \prod_{k<n, k\in Y}[V^>(\alpha_k)] <\prod_{k<n, k\in Y}2^{2^{e(k)+1}} .\]
  By Lemma $5.1$, we get
  \[[V^>(\delta)]\leq 2\delta\leq 2\alpha.\]
   Lemma~\ref{lem:3.2} gives
  \[[V^>(\alpha_n)]=2^{\alpha_n / 2^{n-e(n)-1}}\leq 2^{\alpha/2^{n-e(n)-1}}.\]
  Therefore,
  \[[V^>(\alpha)]\leq   2\alpha  (\prod_{k<n, k\in Y}2^{2^{e(k)+1}}) 2^{\alpha / 2^{n-e(n)-1}}.\]\\
By the definition of sets $V^<$ and $V^>$, we get
  $[V^<(\alpha)]=[V^>(\alpha)]$, so

  \[[V^<(\alpha)][V^>(\beta)]\leq 4(\alpha\beta)(\prod_{k<n, k\in Y}2^{2^{e(k)+1}})^{2}  2^{\frac{\alpha+\beta}{2^{n-e(n)-1}}}.\]\\
Since $\alpha+\beta \leq 2^{n-1}+2^{n-2}$, so $\alpha\beta\leq 2^{2n-2}$. Observe that
$2^{\frac{\alpha+\beta}{2^{n-e(n)-1}}}\leq 2^{\frac{2^{n-1}+2^{n-2}}{2^{n-e(n)-1}}}=2^{2^{e(n)}+2^{e(n)-1}}$,
 and recall that
 $2^{2^{e(n)}+2^{e(n)-1}}=[V(2^{n-1})][V(2^{n-2})]$. It follows that
$$[V^<(\alpha)][V^>(\beta)]\leq 2^{2n}(\prod_{k<n, k\in Y}2^{2^{e(k)+1}})^{2}[V(2^{n-1})][V(2^{n-2})].$$
 Recall that $2^{2^{e(n)-1}}=[V(2^{n-2})]$ and $[V(2^{n-2})]^{2}=[V(2^{n-2})]$. We now see that
 $$[V^<(\alpha)][V^>(\beta)]\leq 2^{2n}(\prod_{k<n, k\in Y}2^{2^{e(k)+1}})^{2}[V(2^{n-1})]^{2}/2^{2^{e(n)-1}}.$$
\end{proof}
\section{Constructing sets $F(2^n)$}\label{ss:F}

 In this section we assume that $r_{i}$ and $Y$, $f_{1}, f_{2}, \ldots $ are as in Theorem $0.1$.
 We moreover assume that there are natural numbers $\{e(i)\}_{i\in Y}$ which  satisfy the following conditions for all $n\in Y$:
 $1\leq e(n)\leq n-1$, sets $S_{n}=\{n-1-e(n), n-1\}$ are disjoint and
 $r_{n}\leq 2^{t(n)}$ where $t(n)=2^{e(n)-1}-3n-4-\sum_{k\in Y, k<n}2^{e(k)+2}.$

 We will construct sets $F(2^{n})\subseteq A(2^{n})$ which let us apply Theorem $4.1$. We begin with the following
lemma, which generalizes Lemma $3.1$ from \cite{sb}.
\begin{lemma} Let $n$ be a natural number.
Suppose that, for all $m<n$, we constructed sets $V(2^{m})$, $U(2^{m})$ which satisfy
properties $1-8$ of Theorem $4.1$, with
$\{e(i)\}_{i\in Y}$ defined as above.
  Consider all $f\in A(k)\cap \{f_{1}, \ldots ,f_{\xi}\}$ with
  $2^n+2^{n-1}\leq k\leq 2^n+2^{n-1}+2^{n-2}$.
  Then there exists a linear $K$-space $F'(2^n)\subseteq A(2^n)$ with the
  following properties:
  \begin{itemize}
  \item $0<\dim F'(2^n)\leq \frac12\dim V(2^{n-1})^2$;
  \item for all $i,j\ge0$ with $i+j=k-2^n$ and for every $f\in A(k)\cap \{f_{1}, \ldots ,f_{\xi}\}$, we have $f\in
    A(i)F'(2^n)A(j)+U^<(i)A(k-i)+A(k-j)U^>(j)$ with the sets
    $U^<(i),U^>(i)$ defined in Section $1$.
  \end{itemize}
 \end{lemma}
 \begin{proof} By Lemma $1.3$, we have $U^{<}(i)\oplus
  V^{<}(i)=A(i)$ and $U^{>}(j)\oplus V^{>}(j)=A(j)$. Therefore,
  $A(i)A(2^n)A(j)= (U^{<}(i)\oplus V^{<}(i))A(2^n)(U^{>}(j)\oplus
  V^{>}(j))$.  Consequently, $A(i+2^n+j)=T'+T$, where $$T=
  U^<(i)A(k-i)+A(k-j)U^>(j), T'=V^{<}(i)A(2^n)V^{>}(j).$$ Hence, $dim A((i+2^n+j))=\dim T+\dim T'-\dim T\cap T'$.
   Observe that $T\cap T'=0$, since
  $\dim A(i+2^n+j)\geq \dim T+ \dim T'$, because $T=U^{<}(i)A(2^n)U^{>}(j)+V^{<}(i)A(2^n)(U^{>}(j)+U^{<}(i))A(2^n)V^{>}(j)$, so $\dim T\leq \dim A(i+2^n+j)-\dim V^{<}(i)A(2^n)V^{>}(j))=\dim A(i+2^n+j)-\dim T'$.
  It follows that $A(i+2^n+j)=T'\oplus T$.

  Consider $f\in A(k)\cap \{f_{1}, \ldots ,f_{\xi}\}$ with $2^n+2^{n-1}\leq k\leq 2^n+2^{n-1}+2^{n-2}$.  We can write $f$ in the form $f={\tilde f}+g$,
  with $g\in T$ and ${\tilde f}\in T'$, where
  \[\tilde f=\sum_{c\in V^<(i),d\in V^>(j)}cz_{c,d,f}d,\qquad z_{c,d,f}\in A(2^n).\]

  Also for the given $f$, we restrict the $c,d$ above to belong to a
  basis, and let $T(i,j,f)\subseteq A(2^n)$ be the subspace spanned by all
  the $z_{c,d,f}$ above. We then have $\dim T(i,j,f)\leq\dim
  V^<(i)\dim V^>(j)$. Observe also $f\in
  A(i)T(i,j,f)A(j)+U^<(i)A(k-i)+A(k-j)U^>(j)$. Define
  \[F'(2^n)=\sum_{k=2^n+2^{n-1}}^{2^n+2^{n-1}+2^{n-2}}\sum_{f\in A(k)\cap \{f_{1}, \ldots ,f_{\xi }\}}\quad \sum_{i+j=k-2^n}T(i,j,f).\]
  We have $2^{n-1}\leq i+j \leq 2^{n-1}+2^{n-2}$, so by Lemma $5.3$
   we have $\dim T(i,j,f)\leq [V^<(\alpha)][V^>(\beta)]\leq
 2^{2n}(\prod_{k<n, k\in Y}2^{2^{e(k)+1}})^{2}[V(2^{n-1})]^{2}/2^{2^{e(n)-1}}.$
 Hence,
 $$\dim F'(2^{n})\leq r_{n}2^{3n}[V(2^{n-1})]^{2}\prod_{k<n, k\in
Y}2^{2^{e(k)+2}}/2^{2^{e(n)-1}}.$$
   To show that
   \[\dim F'(2^n)\leq {1\over 2}\dim V(2^{n-1})^2\]
   it suffices to show that
   $r_{n}2^{3n}\prod_{k<n, k\in
Y}2^{2^{e(k)+2}}\leq {1\over 2}2^{2^{e(n)-1}}$, which follows from assumption $r_{n}\leq 2^{t(n)}$ with $t(n)=2^{e(i)-1}-3n-4-4\sum_{k\in Y, k<n}2^{e(k)}$ from the beginning of this section.
\end{proof}
 In this section we will use the following lemma from \cite{sb}.

\begin{lemma}[Lemma 3.3, \cite{sb}]\label{lem:s1} Let $K$ be an algebraically closed field,
 $n$ be a natural number, and let $T\subseteq A(2^n)$ and
  $Q\subseteq A(2^{n+1})$ be $K$-linear spaces such that $\dim
  T+4\dim Q\leq \dim A(2^n)-2$.
  Then there exists a $K$-linear space $F\subseteq A(2^n)$ of
  dimension at most $\dim A(2^n)-2$ such that $T\subseteq F$ and
  $Q\subseteq FA(2^n)+A(2^n)F$.
\end{lemma}
\begin{proof} A sketch of a proof is included following \cite{sb}.
  Choose a $K$-linear complement $C\subseteq A(2^n)$ to $T$; we have
  \begin{equation}\label{eq:Csum}
    C\oplus T=A(2^n).
  \end{equation}
  Let $\{c_1,\dots,c_s\}$ be a basis of $C$ with $s=\dim
  A(2^n)-\dim T$.

   Let $X, Y$ be two indeterminates.  Let $\eta_t\in K$ and $\zeta_t\in K$, for all
  $t=1,\dots,s$.
   Define a $K$-linear mapping $\overline\Phi\colon C \rightarrow K Y+K Z$
  by $\overline\Phi(c_t)=\eta_tY+\zeta_tZ$ for $t=1,\dots,s$. Using
  $C\oplus T=A(2^n)$, extend it to a mapping $\overline\Phi\colon
  A(2^n)\rightarrow K Y+K Z$ by the condition $T\subseteq\ker
  \overline\Phi$. Using
Hilbert's Nullstellensatz we show that there are assignments
   $\eta_t\in K$ and $\zeta_t\in K$, for all
  $t=1,\dots,s$, such that the following hold.
  \begin{itemize}
  \item [a.] There are $u,v$ such that  $\overline\Phi(c_u)=\eta_uY+\zeta_uZ$ and
  $\overline\Phi(c_v)=\eta_vY+\zeta_vZ$ give two elements that are
  linearly independent over $K$.
  \item[b.] $Q\subseteq A(2^n)\ker(\overline\Phi)+\ker(\overline\Phi)A(2^n)$.
  \end{itemize}
  We define $F:=\ker\overline\Phi.$ Hence $Q\subseteq FA(2^{n})+A(2^{n})F$, as required.
   By construction, we have $T\subseteq\ker\overline\Phi$ so
  $T\subseteq F$ as required. Because
  $\overline\Phi(c_u):=\eta_uY+\zeta_uZ$ and
  $\overline\Phi(c_v):=\eta_vY+\zeta_vZ$ are $K$-linearly
  independent, we have $\dim F\leq \dim A(2^n)-2$ as
  required.
\end{proof}
 The following lemma is a generalisation of Lemma $3.4$ from \cite{sb}.
\begin{lemma}\label{lem:s2}
  Suppose that sets $U(2^m), V(2^m)$ were already constructed for all
  $m<n$, and satisfy the conditions of Theorem~\ref{8props}. Let
  $F=\{f_{1}, f_{2}, \ldots ,f_{\xi}\}$ be as in Theorem $0.1$.
    Define a $K$-linear
  subspace $Q\subseteq A(2^{n+1})$ as follows:
  \[Q=\sum_{f\in F:2^{n}+2^{n-1}\leq \deg f\leq 2^{n}+2^{n-1}+2^{n-2}}\quad
  \sum_{i+j=2^{n+1}-\deg f}V^{>}(i)fV^{<}(j).\]
  Then $\dim Q\leq \frac14(\frac12\dim V(2^{n-1})^2-2)$.
\end{lemma}
\begin{proof}
  By Lemma $5.3$, the inner sum has dimension at most
 \[ 2^{2n}(\prod_{k<n, k\in Y}2^{2^{e(k)+1}})^{2}[V(2^{n-1})]^{2}/2^{2^{e(n)-1}}.\]
  Summing over all $i+j=2^{n+1}-\deg f$
  multiplies by a factor of at most $2^{n}$ (because $2^{n+1}-\deg f\leq 2^{n-1}$); summing over all $f\in
  \{f_{1}, \ldots ,f_{\xi }\}$ with degrees between $2^{n}+2^{n-1}$ and
   $2^{n}+2^{n-1}+2^{n-2}$ multiplies by $r_{n}$.
   Therefore, \[\dim Q \leq r_{n}2^{3n}(\prod_{k<n, k\in
Y}2^{2^{e(k)+1}})^{2}[V(2^{n-1})]^{2}/2^{2^{e(n)-1}}.\]
    By assumption on $r_{n}$ from the beginning of this section, we get
    $\dim Q  \leq \frac1{16}\dim V(2^{n-1})^2$.

  Observe now that $\frac14\dim V(2^{n-1})^2\leq \frac12\dim V(2^{n-1})^2-2$,
  because $V(2^{n-1})^2=(2^{2^{e(n)}})^2\geq 2^{2^2}\geq 16$.  We get
  $\dim Q \leq\frac14(\frac12\dim V(2^{n-1})^2-2)$ as required.
\end{proof}

We are now ready to construct the space $F(2^n)$. Assume
$U(2^m),V(2^m)$ were already constructed for all $m<n$, and satisfy
the conditions of Theorem~\ref{8props}, and suppose that $n\in Y$.
\begin{proposition}[Proposition 3.5, \cite{sb}]\label{prop:F}
Let $K$ be an algebraically closed field.
  With notation as in Lemma $6.1$, there is a linear $K$-space
  $F(2^n)\subseteq A(2^n)$ satisfying $\dim
  F(2^n)\le \dim V(2^{n-1})^2-2$ and \[F'(2^{n})\subseteq
  F(2^{n})+U(2^{n-1})A(2^{n-1})+A(2^{n-1})U(2^{n-1}).\] Moreover, for all $f\in \{f_{1}, \ldots , f_{\xi }\}$ with
  $\deg f\in\{2^n+2^{n-1},\dots,2^n+2^{n-1}+2^{n-2}\}$ we have
  \begin{multline*}
    AfA\cap A(2^{n+1})\subseteq A(2^n)F(2^n)+F(2^n)A(2^n)\\
    +A(2^{n-1})U(2^{n-1})A(2^n)+A(2^n)U(2^{n-1})A(2^{n-1})\\
    +U(2^{n-1})A(2^n+2^{n-1})+A(2^n+2^{n-1})U(2^{n-1}).\\
  \end{multline*}
\end{proposition}

\begin{proof} We outline the proof from \cite{sb}.
  Consider the space $Q\subseteq A(2^{n+1})$ defined in
  Lemma~\ref{lem:s2}, and the space
  $T:=F'(2^n)+U(2^{n-1})A(2^{n-1})+A(2^{n-1})U(2^{n-1})\subseteq
  A(2^n)$, with $F'(2^n)$ as in Lemma $6.1$. Observe that $4\dim
  Q\le\frac12\dim V(2^{n-1})^2-2$ by Lemma~\ref{lem:s2}, hence 
    $\dim T \leq \dim F'(2^n)+(\dim A(2^n)-\dim V(2^{n-1})^2)\leq \dim A(2^n)-\tfrac12\dim V(2^{n-1})^2.$
  Therefore, $\dim T+4\dim Q\leq \dim A(2^n)-2$ and we may apply
  Lemma~\ref{lem:s1} to obtain a set $F$.

 Let $i,j,k\in N$ with $i+j+k=2^{n+1}$, and consider $f\in \{f_{1}, \ldots , f_{\xi }\}$ with $\deg f=k$. By Lemma $1.3$,
    $A(i)fA(j) = (U^>(i)+V^>(i))f(U^<(j)+V^<(j))
    \subseteq V^>(i)fV^<(j) + U^>(i)A(2^{n+1}-i)+A(2^{n+1}-j)U^<(j)$.
     Hence,
   $$ A(i)fA(j) \cap A(2^{n+1})\subseteq Q + U^>(i)A(2^{n+1}-i)+ A(2^{n+1}-j)U^<(j).$$
  By assumption on $k$, we have $i+j\leq 2^{n-1}$, so
  Lemma $1.3$ yields
  $U^>(i)A(2^{n+1}-i)=(U^>(i)A(2^{n-1}-i))A(2^n+2^{n-1})\subseteq
  U(2^{n-1})A(2^n+2^{n-1})$, and similarly
  $A(2^{n+1}-j)U^<(j)\subseteq A(2^n+2^{n-1})U(2^{n-1})$.

  Then, Lemma~\ref{lem:s1} yields $Q\subseteq
  A(2^n)F+FA(2^n)$.
   Consequently, 
  $$ A(i)fA(j) \cap A(2^{n+1})\subseteq   A(2^n)F+FA(2^n) + U(2^{n-1})A(2^n+2^{n-1})+A(2^n+2^{n-1})U(2^{n-1}).$$

  Recall $U(2^{n-1})A(2^{n-1})+A(2^{n-1})U(2^{n-1})\subseteq
  T\subseteq F$.  Let $F(2^{n})\subseteq F$ be a linear $K$-space
  satisfying $F(2^{n})\oplus
  (U(2^{n-1})A(2^{n-1})+A(2^{n-1})U(2^{n-1}))=F$. The last
  claim of the theorem holds when we substitute this equation
   into the above equality.

   Observe that $\dim F(2^{n})= \dim F-\dim
  U(2^{n-1})A(2^{n-1})+A(2^{n-1})U(2^{n-1})$, so $\dim F(2^{n})\leq
  \dim A(2^n)-2-(\dim A(2^n)-\dim V(2^{n-1})^2)\leq \dim V(2^{n-1})^2
  -2$, so the first claim of our theorem holds.  Since
  $F'(2^{n})\subseteq F=F(2^{n})+U(2^{n-1})A(2^{n-1})+A(2^{n-1})U(2^{n-1})$, the proof is finished.
\end{proof}

\section{Free subalgebras and Noetherian images}

 \begin{theorem} Suppose that the assumptions of Theorem $0.1$ hold, and that we use the same notation as in Theorem $0.1$. Assume that for each $n\in Y$ there is a natural number $1\leq e(n)\leq n-1$ such that, for all $n\in Y$, sets $S_{n}=\{n-1-e(n), n-1\}$ are disjoint and  $$r_{n}2^{3n+4}\prod_{k<n, k\in Y}2^{2^{e(k)+2}}\leq 2^{2^{e(n)-1}}.$$
    Then $A/I$ contains a free noncommutative graded subalgebra in two generators, and these generators are monomials of the same degree. In particular, $A/I$ is not Jacobson radical.
    Moreover, $A/I$ can be homomorphically mapped onto a prime, infinite dimensional, Noetherian, graded algebra with linear growth.
 \end{theorem}
\begin{proof} We will first show that $A/I$ contains a free noncommutative subalgebra.   We will construct sets $U(2^{n})$, $V(2^{n})$, $F(2^{n})$ satisfying properties
 $1-7$ and $8'$ from Theorem $4.2$ applied to $e(n)$ as in the assumptions of
   our theorem. The union in~\eqref{def:S} is disjoint by the assumptions. We may therefore start the induction with $U(2^0)=F(2^0)=0$ and
  $V(2^0)=K x+K y$. Then, assuming that we constructed
  $U(2^m),V(2^m)$ for all $m<n$, if $n\in Y$ we construct $F(2^n)$ using
  Proposition $6.1$ and if $n\notin Y$ we set $F(2^{n})=0$. We then construct $U(2^n)$, $V(2^n)$
  using Theorem $4.2$. Let $E$ be defined as in Section $1$.
  By Lemma $1.1$, the set E is an ideal in A and  $A/E$ is an infinite dimensional algebra.
 By Theorem $1.2$, the algebra $A/E$ contains a free noncommutative subalgebra in two generators, and these two generators are monomials of the same degree.
   We will now show that $A/E$ is a homomorphic image of $A/I$.
   We need to show  that $I\subseteq E$, that is that elements $f_{1}, f_{2}, \ldots , f_{\xi } \in E$. Let $f\in A(k)$ be one of these elements for some $2^{n}+2^{n-1}\leq k\leq 2^{n}+2^{n-1}+2^{n-2}$. By Lemma $6.1$ and Proposition $6.1$, we get that  $F'(2^{n})\subseteq F(2^{n})+U(2^{n-1})A(2^{n-1})+A(2^{n-1})U(2^{n-1})\subseteq U(2^{n})$. Consequently, again by Lemma $6.1$ and Proposition $6.1$, we get that $f$ satisfies  the assumptions of Lemma $3.1$. Therefore, and by thesis of Lemma $3.1$, we have $f\in E$, as required.

  We will now prove that $A/ I$ can be mapped onto a prime, graded and infinite dimensional algebra which satisfies a polynomial identity.
  We will construct sets $U(2^{n})$, $V(2^{n})$, $F(2^{n})$ satisfying properties $1-8$ from Theorem $4.1$,
 applied for $e(n)$ as in the assumptions of our theorem. We start the induction with $U(2^0)=F(2^0)=0$ and
  $V(2^0)=K x+K y$. Then, assuming that we constructed
  $U(2^m),V(2^m)$ for all $m<n$, if $n\in Y$ we construct $F(2^n)$ using
  Proposition $6.1$, and if $n\notin Y$ we set $F(2^{n})=0$. We then construct $U(2^n)$, $V(2^n)$
  using Theorem $4.1$ applied for $e(i)$ as in the assumptions. Let $E$ be defined as in Section $1$.
  By Lemma $1.1$, the set E is an ideal in A and  $A/E$ is an infinite dimensional algebra.
 By Theorem $2.5$, the algebra $A/E$  has a graded, prime, Noetherian, infinitely dimensional
 homomorphic image with linear growth; call this $R$. We will now show
  that $R$ is a homomorphic image of $A/I$. We need to show that $I\subseteq E$, that is that elements $f_{1}, f_{2}, \ldots , f_{\xi }\in E$ (these elements are as in  Theorem $0.1$). Let $f\in A(k)$ be one of these elements, for some  $2^{n}+2^{n-1}\leq k\leq 2^{n}+2^{n-1}+2^{n-2}$. By Lemma $6.1$ and Proposition $6.1$, we get that  $F'(2^{n})\subseteq F(2^{n})+U(2^{n-1})A(2^{n-1})+A(2^{n-1})U(2^{n-1})\subseteq U(2^{n})$.
   Consequently, and again by Lemma $6.1$ and Proposition $6.1$, we get that $f$ satisfies  the assumptions of Lemma $3.1$. Therefore, and by thesis of Lemma $3.1$, we have $f\in E$, as required.
  \end{proof}

\begin{lemma} Let $Y$ be a subset of the set of natural numbers and let $\{r_{i}\}_{i\in Y}$ be a sequence of natural numbers which
 satisfy assumptions of Theorem 0.1. Then there are natural numbers $e(i)$ for $i\in Y$ such that
 $1\leq e(n)\leq n-1$ and such that, for all $n\in Y$, sets
 $S_{n}=\{n-1-e(n), n-1\}$ are disjoint and $r_{n}2^{3n+4}\prod_{k<n, k\in
  Y}2^{2^{e(k)+2}}\leq 2^{2^{e(n)-1}}.$
\end{lemma}
\begin{proof} For each $i$, let $e(i)$ be such that
 $2^{2^{e(i)-3}}\leq r_{i}<2^{2^{e(i)-2}}$.
  Note that such $e(i)$ satisfy $e(i)\geq 1$, because $r_{i}\geq 2$.
 By the assumptions,
 $r_{i}<2^{2^{i-j-3}}$ for all $i\in Y$, $j\in Y\cup \{0\}$. Observe
 then that  $e(i)-3< i-j-3$ for all $j<i, j\in Y\cup \{0\}$,
 therefore $e(i)< i-j$, hence $e(i)\leq i-j-1$. This implies
 $e(i)\leq i-1$ and $i-e(i)-1>j-1$. Therefore, sets
 $S(n)=\{n-e(n)-1, n-1\}$ are disjoint for all $n\in Y$.

 We will now show that $r_{n}\leq 2^{t(n)}$, where
  $t(n)=2^{e(i)-1}-3n-4-\sum_{k\in Y, k<n}2^{e(k)+2}.$
  Since $r_{n}<2^{2^{e(n)-2}}$, it suffices to show that
   $2^{e(n)-2}\leq 2^{e(n)-1}-3n-4-\sum_{k\in Y, k<n}2^{e(k)+2}$.
   Hence, it suffices to prove that $$3n+4+\sum_{k\in Y, k<n}2^{e(k)+2}\leq 2^{e(n)-2}.$$
  Since $r_{n}<2^{2^{e(n)-2}}$, it suffices to show that $2^{3n+4+\sum_{k\in Y, k<n}2^{e(k)+2}}\leq r_{n}$.

  Observe first that $\prod_{i<n, i\in Y}2^{2^{e(i)+2}}\leq  \prod_{i<n, i\in Y}r_{i}^{32}$, because by the definition of $e(i)$,  $2^{2^{e(i)+2}}={(2^{2^{e(i)-3}})}^{32}\leq r_{i}^{32}$.
 Hence, it suffices to show that $2^{3n+4}\prod_{i<n, i\in Y}r_{i}^{32}\leq r_{n}$. This follows from the assumptions of Theorem 0.1.
   \end{proof}

{\bf Proof of Theorem 0.1} By Lemma $7.2$, we can find $e(i)$
satisfying the assumptions of Theorem $7.1$, and by the thesis of
 Theorem $7.1$ we get the desired result. The last part of the thesis follows from the Small-Warfield theorem \cite{sw}, which says that prime affine algebras with linear growth are finite dimensional modules over its center.
\\

\medskip

\noindent {\bf Acknowledgements.} The author is grateful to the referees for their close reading of the
 paper and helpful suggestions, to Jason Bell for showing her how to
 prove Lemma 2.4. and to Tom Lenagan and Michael West for many suggestions which helped to improve the paper.

\vspace{3ex}

\begin{minipage}{1.00\linewidth}

\noindent
 Agata Smoktunowicz: \\
Maxwell Institute for Mathematical Sciences\\
School of Mathematics, University of Edinburgh,\\
James Clerk Maxwell Building, King's Buildings, Mayfield Road,\\
Edinburgh EH9 3JZ, Scotland, UK\\
\end{minipage}
\end{document}